\documentclass[11pt,a4paper]{article}

\usepackage[utf8]{inputenc}
\usepackage[T1]{fontenc}
\usepackage[english]{babel}
\usepackage{amsmath}
\usepackage{amsfonts}
\usepackage{amsthm}
\usepackage{amssymb}
\usepackage{makeidx}
\usepackage{graphicx}
\usepackage{lmodern}
\usepackage[left=2cm,right=2cm,top=2cm,bottom=2cm]{geometry}

\usepackage{color}
\usepackage{comment}
\usepackage[dvipsnames]{xcolor}
\usepackage{graphicx}
\usepackage{framed}
\usepackage{tikz}
\usepackage{calrsfs}
\DeclareMathAlphabet{\pazocal}{OMS}{zplm}{m}{n}

\usepackage{fourier-orns}
\usepackage{mathtools}
\usepackage{relsize}
\usepackage{dsfont}
\usepackage{comment}
\usepackage{hyperref}
\newcommand{\B}{\mathbb{B}}  
\newcommand{\R}{\mathbb{R}}

\newcommand{\F}{\pazocal{F}}

\newcommand{\U}{\pazocal{U}}
\newcommand{\K}{\pazocal{K}}
\newcommand{\Vpazo}{\pazocal{V}}
\newcommand{\Wpazo}{\pazocal{W}}
\newcommand{\M}{\pazocal{M}}
\newcommand{\Qpazo}{\pazocal{Q}}
\newcommand{\Hpazo}{\pazocal{H}}
\newcommand{\Npazo}{\pazocal{N}}

\newcommand{\Ppazo}{\pazocal{P}}

\newcommand{\Spazo}{\pazocal{S}}

\newcommand{\Lcal}{\mathcal{L}}
\newcommand{\Vcal}{\mathcal{V}}
\newcommand{\Pcal}{\mathcal{P}}

\newcommand{\Scal}{\mathcal{S}}

\newcommand{\Ucal}{\mathcal{U}}

\newcommand{\Ocal}{\mathcal{O}}

\newcommand{\supp}{\textnormal{supp}}

\newcommand{\Lip}{\textnormal{Lip}}
\newcommand{\AC}{\textnormal{AC}}

\newcommand{\Div}{\textnormal{div}}

\newcommand{\textbn}[1]{\textnormal{\textbf{#1}}}
\newcommand{\co}{\overline{\textnormal{co}} \,}

\newcommand{\yb}{\boldsymbol{y}}

\newcommand{\ub}{\boldsymbol{u}}
\newcommand{\vb}{\boldsymbol{v}}

\newcommand{\INTDom}[3]{\int_{#2} #1 \textnormal{d} #3}
\newcommand{\INTSeg}[4]{\int_{#3}^{#4} #1 \textnormal{d} #2}

\newcommand{\NormC}[3]{\left\| #1  \right\| _ {C^{#2}(#3)}}
\newcommand{\Norm}[1]{\parallel \hspace{-0.1cm} #1 \hspace{-0.1cm} \parallel}

\newtheorem{rmk}{Remark}

\newtheorem{Def}{Definition}
\newtheorem{thm}{Theorem}
\newtheorem{prop}{Proposition}

\newtheorem{taggedhypx}{Hypotheses}
\newenvironment{taggedhyp}[1]
    {\renewcommand\thetaggedhypx{#1}\taggedhypx}
    {\endtaggedhypx}

\title{Mean-Field Optimal Control of Continuity Equations and Differential Inclusions}

\author{Beno\^it Bonnet\thanks{CNRS, IMJ-PRG, UMR 7586, Sorbonne Université, 4 place Jussieu,  75252  Paris,  France. {\tt\small benoit.bonnet@imj-prg.fr}} \hspace{0.05cm} and H\'el\`ene Frankowska\thanks{CNRS, IMJ-PRG, UMR 7586, Sorbonne Université, 4 place Jussieu,  75252  Paris,  France. {\tt\small helene.frankowska@imj-prg.fr}}%
}

\begin{document}

\maketitle


\begin{abstract}
In this article, we propose a new unifying framework for the investigation of multi-agent control problems in the mean-field setting. Our approach is based on a new definition of differential inclusions for continuity equations formulated in the Wasserstein spaces of optimal transport. The latter allows to extend several known results of the classical theory of differential inclusions, and to prove an exact correspondence between solutions of differential inclusions and control systems. We show its appropriateness on an example of leader-follower evacuation problem.
\end{abstract}


\section{Introduction}

The study of self-organisation in large-scale dynamical systems has become a prominent topic in applied mathematics during the course of the past two decades. \textit{Multi-agent systems} appear in an increasingly vast number of applications, ranging from pedestrian dynamics \cite{CPT} and herds analysis \cite{CS1} to fleets of autonomous vehicles \cite{Bullo2009} and opinion formation models \cite{HK}. In this context, control and optimal control problems on multi-agent systems do arise as well. Due to their high underlying dimensionality, the latter are usually studied in the so-called \textit{mean-field approximation} framework, where the discrete collection of agents is replaced by its spatial density. The time evolution of such densities can usually be modelled by means of \textit{non-local continuity equations} of the form
\begin{equation*}
\partial_t \mu(t) + \Div \big( v(t,\mu(t))\mu(t) \big) = 0,
\end{equation*}
where the driving velocity field $v(t,\mu(t))$ depends on the whole density at each time. This type of non-local interactions often takes the form of convolution with interaction kernels, see Section \ref{section:Example} below for an example of such a situation. In this setting, the system is represented by a curve $\mu(\cdot)$ of \textit{probability measures}. Building on far-reaching progresses in the theory of \textit{optimal transport} (see e.g. \cite{AGS}), an important research effort has been directed towards the generalisation of tools of control theory to the metric setting of the space of probability measures. The corresponding contributions include controllability results \cite{Duprez2019}, existence of optimal controls \cite{LipReg,FLOS,MFOC}, optimality conditions \cite{MFPMP,PMPWassConst,PMPWass,Jimenez2020} and numerical methods \cite{Burger2020}. 

In vector spaces, differential inclusions of the form 
\begin{equation*}
\dot x(t) \in F(t,x(t)),
\end{equation*}
have been known to provide a synthetic and powerful way to describe control systems. Indeed under very mild assumptions, any control system can be equivalently rewritten as a differential inclusion. It is then possible to recover many important results of control theory, e.g. controllability results, existence of optimal controls, necessary optimality conditions, regularity properties of the value function, etc., from general properties of the solutions of differential inclusions. We refer the reader e.g. to \cite{Aubin1990} for a detailed study of this topic.

In the context of mean-field control systems, a first approach to differential inclusions in Wasserstein spaces was proposed e.g. in \cite{Jimenez2020,CavagnariMP2018,Cavagnari2018}. However, this formalism did not provide a general one-to-one correspondence between inclusions and control systems, as the controls could depend both on the state $\mu(\cdot)$ of the system and the characteristic curves on which it is supported (for the superposition principle, see e.g. \cite{AmbrosioC2014}). Besides, the idea of considering curves of measures generated by all the \textit{pointwise} solutions of a differential inclusion is less meaningful in terms of Wasserstein geometry than the \textit{functional} approach that we develop here in Section \ref{section:DiffInc}. 

\medskip

The structure of the article is the following. In Section \ref{section:Preli}, we recall classical notions of optimal transport theory and set-valued analysis. In Section \ref{section:DiffInc}, we formulate differential inclusions in Wasserstein spaces and state results on the compactness and closure of their solution sets. We then apply the latter in Section \ref{section:OCP} to a general Mayer optimal control problem with mixed running constraints, which we illustrate on a particular example of evacuation scenario with soft congestion in Section \ref{section:Example}.


\section{Preliminaries}
\label{section:Preli}

In what follows, we recollect known facts about optimal transport and set-valued analysis. We point to the reference monographs \cite{AGS} and \cite{Aubin1990} respectively for a comprehensive introduction to these topics.

Let $\Pcal(\R^d)$ be the set of Borel probability measures over $\R^d$ endowed with the narrow topology induced by 
\begin{equation}
\label{eq:Narrow}
\mu \in \Pcal(\R^d) \mapsto \INTDom{\phi(x)}{\R^d}{\mu(x)}, 
\end{equation}
for all $\phi \in C^0_b(\R^d)$. Here, $(C^0_b(\R^d),\Norm{ \hspace{-0.05cm} \cdot \hspace{-0.05cm}}_{C^0})$ is the set of continuous and bounded functions. We will denote by $\Lip(\phi \, ; \Omega)$ the Lipschitz constant of $\phi(\cdot)$ over $\Omega \subset \R^d$, and by $B(0,R)$ the closed ball of radius $R >0$ centred at zero in $\R^d$. Given $p \in [1,+\infty)$, $L^p(\Omega,\R^d)$ and $W^{1,p}(\Omega,\R^d)$ stand respectively for the Banach spaces of $p$-integrable and Sobolev maps with respect to the standard Lebesgue measure $\Lcal^d$. Define the \textit{momentum} of $\mu \in \Pcal(\R^d)$ by
\begin{equation*}
\M_1(\mu) := \INTDom{|x|}{\R^d}{\mu(x)}. 
\end{equation*}
We henceforth denote by $\Pcal_1(\R^d)$ the set of measures with finite momentum, and by $\Pcal_c(\R^d) \subset \Pcal(\R^d)$ the set of measures whose support
\begin{equation*}
\supp(\mu) := \big\{ x \in \R^d ~\text{s.t.}~ \mu(\Npazo_x) > 0 ~\text{$\forall \Npazo_x$ neigh. of $x$} \big\}
\end{equation*}
is compact. 

Given a measure $\mu \in \Pcal(\R^d)$, define its \textit{pushforward} through a Borel map $f : \R^d \rightarrow \R^d$ by $f_{\#} \mu(B) := \mu(f^{-1}(B))$ for any Borel set $B \subset \R^d$. 

\begin{Def}
Let $\pi^1,\pi^2 : \R^{2d} \rightarrow \R^d$ be the projection operations onto the first and second factor. A measure $\gamma \in \Pcal(\R^{2d})$ is a \textit{transport plan} between $\mu,\nu \in \Pcal(\R^d)$ if $\pi^1_{\#} \gamma = \mu$ and $\pi^2_{\#} \gamma = \nu$. The set of all transport plans is denoted by $\Gamma(\mu,\nu)$. 
\end{Def}

\begin{Def}
The \textit{Wasserstein distance} $W_1(\mu,\nu)$ between two measures $\mu,\nu \in \Pcal_1(\R^d)$ is defined by 
\begin{equation*}
W_1(\mu,\nu) := \min_{\gamma \in \Gamma(\mu,\nu)} \INTDom{|x-y|}{\R^{2d}}{\gamma(x,y)}.
\end{equation*}
It is known that the $W_1$-topology metrises the narrow topology induced by \eqref{eq:Narrow}, in the sense that 
\begin{equation*}
W_1(\mu_N,\mu) \underset{N \rightarrow +\infty}{\longrightarrow} 0 ~\Longleftrightarrow~ \left\{
\begin{aligned}
\mu_N & \underset{N \rightarrow +\infty}{~\rightharpoonup^*} \mu, \\
\M_1(\mu_N) & \underset{N \rightarrow +\infty}{\longrightarrow} \M_1(\mu). 
\end{aligned}
\right.
\end{equation*}
\end{Def}

We say that an absolutely continuous curve of measures $\mu(\cdot) \in \AC([0,T],\Pcal_c(\R^d))$ solves the \textit{continuity equation} driven by the velocity field $(t,x) \mapsto v(t,x)$ with initial condition $\mu^0 \in \Pcal_c(\R^d)$, which writes
\begin{equation}
\label{eq:ContinuityEquations}
\left\{
\begin{aligned}
& \partial_t \mu(t) + \Div \big( v(t)\mu(t) \big) = 0, \\
& \mu(0) = \mu^0, 
\end{aligned}
\right.
\end{equation}
where ``$\Div$'' stands for the distributional divergence, provided that $\mu(0) = \mu^0$ and
\begin{equation}
\label{eq:ContinuityEquationsDistrib}
\INTSeg{\INTDom{\big( \partial_t \phi(t,x) + \langle \nabla_x \phi(t,x) , v(t,x) \rangle \big)}{\R^d}{\mu(t)(x)}}{t}{0}{T}= 0,
\end{equation}
for any $\phi \in C^{\infty}_c([0,T] \times \R^d)$.  

\begin{taggedhyp}{\textbn{(C1)}}
\label{hyp:C1}
The velocity field $v(\cdot,\cdot)$ is a \textit{Carathéodory vector field}, i.e. the map $t \mapsto v(t,x)$ is $\Lcal^1$-measurable for all $x \in \R^d$ and the map $x \mapsto v(t,x)$ is continuous for $\Lcal^1$-almost every $t \in [0,T]$. Moreover, there exists $m \in L^1([0,T],\R_+)$ such that
\begin{equation}
\label{eq:SubLin}
|v(t,x)| \leq m(t) \big( 1 + |x| \big), 
\end{equation}
for $\Lcal^1$-almost every $t \in [0,T]$ and all $x \in \R^d$.
\end{taggedhyp}

\begin{taggedhyp}{\textbn{(C2)}} 
\label{hyp:C2}
For any compact set $K \subset \R^d$, there exists $l_K \in L^1([0,T],\R_+)$ such that
\begin{equation}
\label{eq:LipBound}
\Lip(v(t,\cdot) \, ; K) \leq l_K(t),
\end{equation}
for $\Lcal^1$-almost every $t \in [0,T]$.
\end{taggedhyp}

We now recall a classical well-posedness result for continuity equations (see e.g. \cite{AmbrosioC2014}).

\begin{thm}
Let $r >0$, $\mu^0 \in \Pcal(B(0,r))$ and suppose that $(t,x) \mapsto v(t,x)$ satisfies \ref{hyp:C1}. Then, there exists a solution $\mu(\cdot)$ of \eqref{eq:ContinuityEquations} starting from $\mu^0$. Moreover, there exist $R_r >0$ and $m_r \in L^1([0,T],\R_+)$ depending only on $r,m$ such that for each such solution, it holds 
\begin{equation}
\label{eq:Unif_SuppLip}
\supp(\mu(t)) \subset B(0,R_r),~~ W_1(\mu(t),\mu(s)) \leq \mathsmaller{\INTSeg{m_r(\tau)}{\tau}{s}{t}},
\end{equation}
for all times $0 \leq s \leq t \leq T$. If $(t,x) \mapsto v(t,x)$ also satisfies \ref{hyp:C2}, then the curve $\mu(\cdot)$ is unique.
\end{thm}

Let $(\Scal,d_{\Scal})$ be a complete separable metric space and $(X,\Norm{\cdot}_X)$ be a separable Banach space. Given a subset $B \subset X$, we denote by $\co(B)$ its \textit{closed convex hull}. We say that a \textit{set-valued map} $\F : \Scal \rightrightarrows X$ has \textit{closed values} if $\F(s) \subset X$ is a closed set for all $s \in \Scal$. Furthermore, we say that $\F : [0,T] \rightrightarrows \Scal$ is $\Lcal^1$-measurable if the sets
\begin{equation*}
\F^{-1}(\Ocal) := \big\{ t \in [0,T] ~\text{s.t.}~ \F(t) \cap \Ocal \neq \emptyset \big\},
\end{equation*}
are $\Lcal^1$-measurable for any open subset $\Ocal \subset \Scal$.  

\begin{thm}
\label{thm:Selection}
Suppose that $\F : [0,T] \rightrightarrows \Scal$ is measurable with closed values. Then, there exists a \textit{measurable selection} $f : [0,T] \rightarrow \Scal$ in $\F(\cdot)$, i.e. $f(t) \in \F(t)$ for $\Lcal^1$-almost every $t \in [0,T]$.
\end{thm}


\section{Differential inclusions in Wasserstein spaces}
\label{section:DiffInc}

In this section, we propose a functional approach to differential inclusions in Wasserstein spaces and state the compactness theorem and the Relaxation theorem.

\begin{Def}
\label{def:WassInc}
Given $V : [0,T] \times \Pcal_c(\R^d) \rightrightarrows C^0(\R^d,\R^d)$, a curve of measures $\mu(\cdot)$ solves the \textit{differential inclusion}
\begin{equation}
\label{eq:WassInc}
\left\{
\begin{aligned}
& \partial_t \mu(t) \in - \Div \big( V(t,\mu(t)) \mu(t) \big), \\
& \mu(0) = \mu^0, 
\end{aligned}
\right.
\end{equation}
if there exists a measurable selection $t \in [0,T] \mapsto v(t) \in V(t,\mu(t))$ such that $\mu(\cdot)$ solves \eqref{eq:ContinuityEquations}. The couple $(\mu(\cdot),v(\cdot))$ is then called a \textit{trajectory-selection} pair for \eqref{eq:WassInc}.
\end{Def}

\begin{Def}
For any compact set $K \subset \R^d$, define
\begin{equation*}
V_K(t,\mu) := \big\{ v_{|K} ~\text{s.t.}~ v \in V(t,\mu) \big\} \subset C^0(K,\R^d), 
\end{equation*}
for any $(t,\mu)$, where $v_{|K}$ is the restriction of $v$ to $K$.
\end{Def}

\begin{taggedhyp}{\textbn{(D)}} 
\label{hyp:D}
Suppose that for any compact subset $K \subset \R^d$, the following holds.  
\begin{enumerate}
\item[$(i)$] The map $t \in [0,T] \rightrightarrows V_K(t,\mu)$ is measurable with closed non-empty images for all $\mu \in \Pcal_c(\R^d)$.
\item[$(ii)$] There exists $m \in L^1([0,T],\R_+)$ such that for $\Lcal^1$-almost every $t \in [0,T]$ and all $\mu \in \Pcal_c(\R^d)$, any element $v \in V(t,\mu)$ satisfies for all $x \in \R^d$
\begin{equation*}
|v(x)| \leq m(t) \big( 1 + |x| + \M_1(\mu) \big). 
\end{equation*}
\item[$(iii)$] There exists $l_K \in L^1([0,T],\R_+)$ such that for $\Lcal^1$-almost every $t \in [0,T]$ and all $\mu \in \Pcal(K)$, any element $v \in V(t,\mu)$ satisfies 
\begin{equation*}
\Lip(v \, ; K) \leq l_K(t). 
\end{equation*}
\item[$(iv)$] There exists $L_K \in L^1([0,T],\R_+)$ such that for $\Lcal^1$-almost every $t \in [0,T]$ and all $\mu,\nu \in \Pcal(K)$ it holds
\begin{equation*}
V_K(t,\nu) \subset V_K(t,\mu) + L_K(t)W_1(\mu,\nu) \B_{C^0(K,\R^d)},
\end{equation*}
where $\B_{C^0(K,\R^d)}$ denotes the closed unit ball in $(C^0(K,\R^d),\NormC{\cdot}{0}{K,\R^d})$.
\end{enumerate}
\end{taggedhyp}

The following existence result can be proven under hypotheses \ref{hyp:D} by an iterative scheme in the spirit e.g. of \cite[Chapter 2.3.13]{Vinter}. The proof of this result is fairly lengthy and technical and will appear elsewhere. 

\begin{thm}
\label{thm:Existence}
Let $V : [0,T] \times \Pcal_c(\R^d) \rightrightarrows C^0(\R^d,\R^d)$  be verifying \ref{hyp:D}. Then for any $\mu^0 \in \Pcal_c(\R^d)$, there exists a solution $\mu(\cdot)$ of \eqref{eq:WassInc}. Moreover for every $r > 0$, there exist $R_r  > 0$ and $m_r \in L^1([0,T],\R_+)$ such that any solution of \eqref{eq:WassInc} with $\mu^0 \in \Pcal(B(0,r))$ satisfies \eqref{eq:Unif_SuppLip}.
\end{thm}


\subsection{Compactness of the solution set}

An essential property of \eqref{eq:WassInc} that we shall use in the sequel is the compactness of its solution set whenever $V : [0,T] \times \Pcal_c(\R^d) \rightrightarrows C^0(\R^d,\R^d)$ has convex values.

\begin{thm}
\label{thm:Compactness}
Suppose that $V : [0,T] \times \Pcal_c(\R^d) \rightrightarrows C^0(\R^d,\R^d)$ has \textit{convex values} and that it satisfies \ref{hyp:D}. Then, the set of all trajectories of \eqref{eq:WassInc} is compact in the topology of the uniform convergence. 
\end{thm}

\begin{proof}
Let $r>0$ be such that $\mu^0 \in \Pcal(B(0,r))$ and $(\mu_N(\cdot),v_N(\cdot))$ be a sequence of trajectory-selection pairs for \eqref{eq:WassInc}. By Theorem \ref{thm:Existence}, there exists $R_r >0$ and $m_r \in L^1([0,T],\R_+)$ such that $\mu_N(\cdot)$ satisfies \eqref{eq:Unif_SuppLip} for all $N \geq 1$. Whence by the Ascoli-Arzelà Theorem, there exists a curve $\mu^*(\cdot) \in \AC([0,T],\Pcal_c(\R^d))$ such that
\begin{equation}
\label{eq:UnifConvMeas}
\sup_{t \in [0,T]} W_1(\mu_N(t),\mu^*(t)) \underset{N \rightarrow +\infty}{\longrightarrow} 0,
\end{equation}
along an unrelabelled subsequence. 

Define the closed ball $K := B(0,R_r)$. By \ref{hyp:D}-$(ii),(iii)$ and \eqref{eq:Unif_SuppLip}, there exists $c \in L^1([0,T],\R_+)$ such that 
\begin{equation*}
\Norm{v_N(t,\cdot)}_{W^{1,p}(K,\R^d)} \leq c(t), 
\end{equation*}
for $\Lcal^1$-almost every $t \in [0,T]$ and all $p \in (1,+\infty)$. Thus by the Dunford-Pettis theorem (see e.g. \cite[Thm. 1.38]{AmbrosioFuscoPallara}) applied to Bochner integrable maps (see e.g. \cite{DiestelUhl}), the sequence $(v_N(\cdot))$ admits a weakly-converging subsequence towards some $v^*(\cdot)$ in $L^1([0,T],W^{1,p}(K,\R^d))$. By choosing $p >d$ and invoking Morrey's embedding (see e.g. \cite[Thm. 9.12]{Brezis}), it holds for any $\phi \in C^{\infty}_c(\R^d)$
\begin{equation}
\label{eq:IntegralConvVel}
\INTSeg{\INTDom{\langle v^*(t,x) - v_N(t,x) , \nabla \phi(x) \rangle}{\R^d}{\mu^*(t)}}{t}{0}{T} \underset{N \rightarrow +\infty}{\longrightarrow} 0,
\end{equation}
along an unrelabelled subsequence. Plugging \eqref{eq:UnifConvMeas}-\eqref{eq:IntegralConvVel} into \eqref{eq:ContinuityEquationsDistrib} in turn yields that $(\mu^*(\cdot),v^*(\cdot))$ solves \eqref{eq:ContinuityEquations}. 

We now prove that $v^*(t) \in V(t,\mu^*(t))$ for $\Lcal^1$-almost every $t \in [0,T]$. By hypothesis \ref{hyp:D}-$(iv)$ and \eqref{eq:UnifConvMeas}, there exists $(\tilde{v}_N(\cdot))$ such that 
\begin{equation}
\label{eq:TildeV}
\NormC{v_N(t) - \tilde{v}_N(t)}{0}{K,\R^d} \underset{N \rightarrow +\infty}{\longrightarrow} 0,~~ \tilde{v}_N(t) \in V(t,\mu^*(t)),
\end{equation}
for $\Lcal^1$-almost every $t \in [0,T]$. By repeating the same argument as before, $(\tilde{v}_N(\cdot))$ admits a weakly-converging subsequence in $L^1([0,T],W^{1,p}(K,\R^d))$. Moreover, observe that $\tilde{v}_N(\cdot) \in \Vcal_K$ where 
\begin{equation*}
\Vcal_K = \big\{ v \in L^1([0,T],C^0(K,\R^d)) ~\text{s.t.}~ v(t) \in V(t,\mu^*(t)) \big\}.
\end{equation*}
It can be shown (see e.g. \cite{MFOC}) that $\Vcal_K$ is closed in the strong $L^1([0,T],W^{1,p}(K,\R^d))$-topology. Also, $\Vcal_K$ is convex since $V(\cdot,\cdot)$ has convex values, so that it is weakly closed by Mazur's Lemma (see e.g. \cite[Thm. 3.7]{Brezis}). Thus, by \eqref{eq:TildeV}, we conclude that $v^*(t) \in V(t,\mu^*(t))$ for $\Lcal^1$-almost every $t \in [0,T]$, which proves our claim. 
\end{proof}


\subsection{Relaxation theorem and value functions}

In the case where the values of $V : [0,T] \times \Pcal_c(\R^d) \rightrightarrows C^0(\R^d,\R^d)$ are not convex, the closure of the set of trajectories can be characterised by the \textit{Relaxation theorem}. 

\begin{Def}
We define the \textnormal{closed convex hull} of a set $V \subset C^0(\R^d,\R^d)$ as
\begin{equation}
\label{eq:ConvHull}
\co V := \{ v ~\text{s.t.}~ v_{\vert K} \in \co V_K ~ \forall K \subset \R^d \, \text{compact} \big\}, 
\end{equation}
where $\co V_K$ is the closed convex-hull of $V_K(\mu)$ in the Banach space $(C^0(K,\R^d),\Norm{\cdot}_{C^0})$. 

\end{Def}

\begin{thm}
\label{thm:RelaxationWass}
Suppose that $V : [0,T] \times \Pcal_c(\R^d) \rightrightarrows C^0(\R^d,\R^d)$ satisfies \ref{hyp:D}. Then for any $\delta > 0$, $\mu^0 \in \Pcal_c(\R^d)$ and any solution $\mu(\cdot) \in \AC([0,T],\Pcal_c(\R^d))$ of 
\begin{equation}
\label{eq:RelaxedInc_Theorem}
\left\{
\begin{aligned}
& \partial_t \mu(t) \in - \Div \big( \co V(t,\mu(t)) \mu(t) \big), \\
& \mu(0) = \mu^0, 
\end{aligned}
\right.
\end{equation}
there exists a solution $\mu_{\delta}(\cdot) \in \AC([0,T],\Pcal_c(\R^d))$ of
\begin{equation}
\label{eq:Inc_Theorem}
\left\{
\begin{aligned}
& \partial_t \mu_{\delta}(t) \in - \Div \big( V(t,\mu_{\delta}(t)) \mu_{\delta}(t) \big), \\
& \mu_{\delta}(0) = \mu^0, 
\end{aligned}
\right.
\end{equation}
such that 
\begin{equation*}
\sup_{t \in [0,T]} W_1(\mu(t),\mu_{\delta}(t)) \leq \delta.
\end{equation*}
\end{thm}

\begin{proof}
The proof of this result cannot be fully detailed here due to its complexity and the lack of space, and will thus be published elsewhere. Nonetheless, we wish to sketch its main arguments. Let $t \in [0,T] \mapsto v(t) \in V(t,\mu(t))$ be a selection generating a solution $\mu(\cdot)$ of \eqref{eq:RelaxedInc_Theorem} in the sense of Definition \ref{def:WassInc}.

The first step of the proof is to build a curve of measures $\nu(\cdot)$ solution of
\begin{equation}
\label{eq:Intermediate_Inclusion}
\left\{
\begin{aligned}
& \partial_t \nu(t) \in - \Div \big( V(t,\mu(t)) \nu(t) \big), \\
& \nu(0) = \mu^0, 
\end{aligned}
\right.
\end{equation}
such that $W_1(\mu(t),\nu(t))$ is small. To do so, one chooses an adequate subdivision $\{ [t_i,t_{i+1}]\}_{i=0}^{N-1}$ of $[0,T]$ and applies Aumann's theorem (see e.g. \cite[Thm. 8.6.4]{Aubin1990}) to recover the existence of measurable selections $t \in [t_i,t_{i+1}] \mapsto v_i(t) \in V(t,\mu(t))$ whose integrals are sufficiently close to that of $t \in [t_i,t_{i+1}] \mapsto v(t) \in \co V(t,\mu(t))$. One can then show that the curve $\nu(\cdot)$ driven by $w(t) := \sum_{i=0}^{N-1} \mathds{1}_{[t_i,t_{i+1}]}(t) v_i(t)$ solves \eqref{eq:Intermediate_Inclusion} and is close to $\mu(\cdot)$.

The second step is to apply a version of Filippov's estimate (see e.g. \cite[Thm 2.3.13]{Vinter}) adaptated to Wasserstein spaces, which together with the Lipschitzianity \ref{hyp:D}-$(iv)$ of $\mu \mapsto V(t,\mu)$ allows to recover the existence of a curve $\mu_{\delta}(\cdot)$ solution of \eqref{eq:Inc_Theorem} such that $W_1(\mu_{\delta}(t),\nu(t))$ is small. The result then follows from the triangle inequality.
\end{proof}

The next proposition provides an application of the Relaxation theorem to investigate the value function of optimal control problems on differential inclusions.

\begin{prop}
Suppose that $V : [0,T] \times \Pcal_c(\R^d) \rightrightarrows C^0(\R^d,\R^d)$ satisfies \ref{hyp:D} and that $\varphi : \Pcal_c(\R^d) \rightarrow \R$ is continuous in the $W_1$-metric. Then, the \textnormal{value functions} $\Vcal,\Vcal_{\co} : [0,T] \times\Pcal_c(\R^d) \rightarrow \R$ defined respectively by 
\begin{equation*}
\Vcal \big( \tau , \mu_{\tau} \big) := \left\{ 
\begin{aligned}
\inf_{\mu(\cdot)} \, & \big[ \varphi(\mu(T)) \big] \\
\textnormal{s.t.} & 
\left\{ 
\begin{aligned}
& \partial_t \mu(t) \in - \Div \Big( V(t,\mu(t)) \mu(t) \Big), \\
& \mu(\tau) = \mu_{\tau}, 
\end{aligned}
\right.
\end{aligned}
\right.
\end{equation*}
and
\begin{equation*}
\Vcal_{\co} \big( \tau , \mu_{\tau} \big) := \left\{ 
\begin{aligned}
\inf_{\mu(\cdot)} \, & \big[ \varphi(\mu(T)) \big] \\
\textnormal{s.t.} & 
\left\{ 
\begin{aligned}
& \partial_t \mu(t) \in - \Div \Big( \co V(t,\mu(t)) \mu(t) \Big), \\
& \mu(\tau) = \mu_{\tau}, 
\end{aligned}
\right.
\end{aligned}
\right.
\end{equation*}
for every $(\tau,\mu_{\tau}) \in [0,T] \times \Pcal_c(\R^d)$ are equal, where $\co V(t,\mu)$ is defined as in \eqref{eq:ConvHull}.
\end{prop}

\begin{proof}
Let $(\tau,\mu_{\tau}) \in [0,T] \times \Pcal_c(\R^d)$ and $r_{\tau} > 0$ be such that $\mu_{\tau} \in \Pcal(B(0,r_{\tau}))$. Let $K_{\tau} := B(0,R_{r_{\tau}})$ be given by Theorem \ref{thm:Existence}, i.e. it is such that all the trajectories of the differential inclusions \eqref{eq:RelaxedInc_Theorem}-\eqref{eq:Inc_Theorem} have support in $K_{\tau}$. Since every trajectory of \eqref{eq:Inc_Theorem} is also a trajectory of the relaxed inclusion \eqref{eq:RelaxedInc_Theorem}, it holds 
\begin{equation}
\label{eq:ValueFunction}
\Vcal_{\co}(\tau,\mu_{\tau}) \leq \Vcal(\tau,\mu_{\tau}). 
\end{equation}
Conversely, let $\mu^*(\cdot)$ be a trajectory of \eqref{eq:RelaxedInc_Theorem}. By Theorem \ref{thm:RelaxationWass}, there exists a sequence $(\mu_N(\cdot)) \subset \AC([0,T],\Pcal(K_{\tau}))$ of solution of \eqref{eq:Inc_Theorem} such that 
\begin{equation*}
\sup_{t \in [\tau,T]} W_1(\mu_N(t),\mu^*(t)) \underset{N \rightarrow +\infty}{\longrightarrow} 0.
\end{equation*}
Recalling that $\varphi(\cdot)$ is continuous, we deduce that for every $\epsilon > 0$ there exists an integer $N_{\epsilon} \geq 1$, such that 
\begin{equation}
\label{eq:EpsilonIneq}
\Vcal(\tau,\mu_{\tau}) \leq \varphi(\mu_N(T)) \leq \varphi(\mu^*(T)) + \epsilon,
\end{equation}
for every $N \geq N_{\epsilon}$. Thus taking the infimum over the trajectories $\mu^*(\cdot)$ of \eqref{eq:RelaxedInc_Theorem} in \eqref{eq:EpsilonIneq}, we recover that 
\begin{equation*}
\Vcal(\tau,\mu_{\tau}) \leq \Vcal_{\co}(\tau,\mu_{\tau}) + \epsilon, 
\end{equation*}
for every $\epsilon > 0$, which together with \eqref{eq:ValueFunction} yields that $\Vcal(\tau,\mu_{\tau}) = \Vcal_{\co}(\tau,\mu_{\tau})$. 
\end{proof}


\section{Existence of mean-field optimal controls for a Mayer problem}
\label{section:OCP}

In this section, we apply the set-theoretic tools of Section \ref{section:DiffInc} to study the Mayer optimal control problem
\begin{equation*}
(\Ppazo) ~ \left\{
\begin{aligned}
\min_{u(\cdot) \in \U} & \big[ \varphi(\mu(T)) \big] \\
\text{s.t.} ~~ & \left\{
\begin{aligned}
& \partial_t \mu(t) + \Div \big( v(t,\mu(t),u(t)) \mu(t) \big) = 0, \\ 
& \mu(0) = \mu^0, 
\end{aligned}
\right. \\
\text{and} ~~ & \left\{
\begin{aligned}
\mu(t) & \in \K(t) ~~ \text{for all times $t \in [0,T]$},\\
\mu(T) & \in \Qpazo_T.
\end{aligned}
\right.
\end{aligned}
\right.
\end{equation*}
The minimisation in $(\Ppazo)$ is taken over the set of admissible controls $\U$ of measurable selections $t \in [0,T] \mapsto u(t) \in U(t,\mu(t))$ where $U(\cdot,\cdot)$ is a family of compact subsets in the metric space $(\Ucal,d_{\Ucal})$, and
\begin{equation*}
v : (t,\mu,u) \in [0,T] \times \Pcal_c(\R^d) \times \Ucal \rightarrow C^0(\R^d,\R^d),
\end{equation*}
is a controlled non-local vector field. We define the set-valued map $V : [0,T] \times \Pcal_c(\R^d) \rightrightarrows C^0(\R^d,\R^d)$ by 
\begin{equation}
\label{eq:SetValuedControlDyn}
V(t,\mu) := v \big( t,\mu, U(t,\mu) \big),
\end{equation}
for $\Lcal^1$-almost every $t \in [0,T]$ and all $\mu \in \Pcal_c(\R^d)$. 

\begin{taggedhyp}{\textbn{(OCP)}}
\label{hyp:OCP}
Suppose that for any compact set $K \subset \R^d$, the following holds.
\begin{enumerate}
\item[$(i)$] The set-valued map $(t,\mu) \mapsto U(t,\mu) \subset \Ucal$ is $\Lcal^1$-measurable in $t \in [0,T]$ and continuous with respect to $\mu \in \Pcal_c(\R^d)$ in the $W_1$-metric.
\item[$(ii)$] The map $(t,\mu,u,x) \mapsto v(t,\mu,u)(x) \in \R^d$ is $\Lcal^1$-measurable with respect to $t \in [0,T]$ and continuous with respect to $(u,x) \in \Ucal \times \R^d$. Moreover, there exists $m \in L^1([0,T],\R_+)$ such that for $\Lcal^1$-almost every $t \in [0,T]$ it holds
\begin{equation*}
\big| v(t,\mu,u)(x) \big| \leq m(t) \big( 1 + |x| + \M_1(\mu) \big),
\end{equation*}
for all $(\mu,u,x) \in \Pcal_c(\R^d) \times \Ucal \times \R^d$. 

There exist two maps $l_K,L_K \in L^1([0,T],\R_+)$ such that for $\Lcal^1$-almost every $t \in [0,T]$, we have
\begin{equation*}
\Lip(v(t,\mu,u)\, ; K) \leq l_K(t), 
\end{equation*}
for any $(\mu,u) \in \Pcal(K) \times \Ucal$, and 
\begin{equation*}
\big\| v (t,\mu,u) - v(t,\nu,u) \big\|_{C^0(K,\R^d)} \leq L_K(t) W_1(\mu,\nu), 
\end{equation*}
for any $\mu,\nu \in \Pcal(K)$ and $u \in \Ucal$.  
\item[$(iii)$] The set-valued map $(t,\mu) \in [0,T] \times \Pcal_c(\R^d) \rightrightarrows V(t,\mu)$ has convex values.
\item[$(iv)$] The final cost $\mu \in \Pcal_c(\R^d) \mapsto \varphi(\mu) \in \R $ is lower-semicontinuous over $\Pcal(K)$ in the $W_1$-metric. 
\item[$(v)$] The running and final constraint sets $t \in [0,T] \mapsto \K(t) \subset \Pcal_c(\R^d)$ and $\Qpazo_T \subset \Pcal_c(\R^d)$ are $W_1$-closed.
\end{enumerate}
\end{taggedhyp}

\begin{rmk}
Hypotheses \ref{hyp:OCP}-$(ii)$ are a natural extension of the classical assumptions of control theory to Wasserstein spaces, guaranteeing  that to every measurable control corresponds a unique solution defined on the whole time interval $[0,T]$ of the controlled continuity equation. Hypotheses \ref{hyp:OCP}-$(iii)$, $(iv)$ and $(v)$ are classical for  the existence of optimal solutions.  Assumption \ref{hyp:OCP}-$(i)$ is needed to link the trajectories of an associated differential inclusion to those of  the control system.  When the sets $U(\cdot,\cdot)$ are independent from both time and states, it simply means that measurable controls take their values in a compact set. 
\end{rmk}

First, we prove that the set of solutions of the control system driving $(\Ppazo)$ coincides with the solution set of the differential inclusion with $V(\cdot,\cdot)$ defined by \eqref{eq:SetValuedControlDyn}.

\begin{prop}
\label{prop:DiffInc_Control}
Suppose that \ref{hyp:OCP}-$(i)$ and \ref{hyp:OCP}-$(ii)$ hold, and that the set-valued map $V : [0,T] \times \Pcal_c(\R^d) \rightrightarrows C^0(\R^d,\R^d)$ is defined as in \eqref{eq:SetValuedControlDyn}. Then, a curve $\mu(\cdot)$ solves \eqref{eq:WassInc} if and only if it is a solution of the controlled dynamics generated by some $u(\cdot) \in \U$.
\end{prop}

\begin{proof}
Let $\mu(\cdot) \in \AC([0,T],\Pcal_c(\R^d))$ be an admissible curve for $(\Ppazo)$ driven by a control map $u(\cdot) \in \U$. By construction, the time-dependent velocity field $\vb : t \in [0,T] \mapsto v(t,\mu(t),u(t)) \in C^0(\R^d,\R^d)$ is such that $\vb(t) \in V(t,\mu(t))$ for $\Lcal^1$-almost every $t \in [0,T]$, so that $\mu(\cdot)$ solves \eqref{eq:WassInc} with $V(\cdot,\cdot)$ as in \eqref{eq:SetValuedControlDyn}.

Conversely, suppose that $\mu(\cdot) \in \AC([0,T],\Pcal_c(\R^d))$ is a solution of \eqref{eq:WassInc} driven by $V(\cdot,\cdot)$. Notice that as a consequence of hypothesis \ref{hyp:OCP}-$(ii)$, the set valued map $V(\cdot,\cdot)$ defined in \eqref{eq:SetValuedControlDyn} satisfies hypotheses \ref{hyp:D}. Let $K \subset \R^d$ be the compact set given by Theorem \ref{thm:Existence}, i.e. $\supp(\mu(t)) \subset K$ for all times $t \in [0,T]$. By Definition \ref{def:WassInc}, there exists a measurable selection $t \in [0,T] \mapsto \vb(t) \in V_K(t,\mu(t))$ such that 
\begin{equation*}
\partial_t \mu(t) + \Div \big( \vb(t) \mu(t) \big) = 0. 
\end{equation*}
Moreover, observe that
\begin{equation*}
V_K(t,\mu(t)) = v_{|K} \big(t,\mu(t),U(t,\mu(t)) \big) \subset C^0(K,\R^d),
\end{equation*}
where $(t,u) \mapsto v_{|K}(t,\mu(t),u)$ is $\Lcal^1$-measurable with respect to $t \in [0,T]$ and continuous with respect to $u \in \Ucal$. Thus, we can apply Theorem \ref{thm:Selection} and \cite[Thm. 8.2.8]{Aubin1990} to recover the existence of a measurable selection $t \in [0,T] \mapsto u(t) \in U(t,\mu(t))$ such that $\vb(t) = v_{|K}(t,\mu(t),u(t))$ for $\Lcal^1$-almost every $t \in [0,T]$. Whence, $\mu(\cdot)$ is admissible for $(\Ppazo)$ and generated by the control map $t \in [0,T] \mapsto u(t) \in U(t,\mu(t))$.
\end{proof}

Using the above result, we can show that under hypotheses \ref{hyp:OCP}, problem $(\Ppazo)$ admits a solution.

\begin{thm}
Suppose that hypotheses \ref{hyp:OCP} hold. Then, there exists an optimal trajectory-control pair $(\mu(\cdot),u(\cdot)) \in \AC([0,T],\Pcal_c(\R^d)) \times \U$ for $(\Ppazo)$. 
\end{thm}

\begin{proof}
Let $(u_N(\cdot))$ be a minimising sequence for $(\Ppazo)$ and $(\mu_N(\cdot))$ be the corresponding admissible trajectories. By \ref{hyp:OCP}-$(ii)$, we know that $V(\cdot,\cdot)$ defined in \eqref{eq:SetValuedControlDyn} satisfies \ref{hyp:D}, and that is has convex values by \ref{hyp:OCP}-$(iii)$. Thus by Theorem \ref{thm:Compactness}, there exists a trajectory-selection pair $(\mu^*(\cdot),\vb^*(\cdot))$ solution of \eqref{eq:WassInc} such that 
\begin{equation}
\label{eq:ConvW1}
\sup_{t \in [0,T]} W_1(\mu_N(t),\mu^*(t)) \underset{N \rightarrow +\infty}{\longrightarrow} 0,
\end{equation}
along an unrelabelled subsequence. Furthermore by Proposition \ref{prop:DiffInc_Control}, there exists an admissible control $t \in [0,T] \mapsto u^*(t) \in U(t,\mu^*(t))$ such that $\vb(t) = v(t,\mu^*(t),u^*(t))$ for $\Lcal^1$-almost every $t \in [0,T]$. 

Since $\K(t)$ and $\Qpazo_T$ are closed in the $W_1$-topology by \ref{hyp:OCP}-$(v)$, we have by \eqref{eq:ConvW1} that for all times $t \in [0,T]$
\begin{equation*}
\mu^*(t) \in \K(t) \quad \text{and} \quad \mu^*(T) \in \Qpazo_T.
\end{equation*}
Moreover, $\varphi(\cdot)$ is lower-semicontinuous in the $W_1$-metric as a consequence of \ref{hyp:OCP}-$(iv)$, so that it holds
\begin{equation*}
\varphi(\mu^*(T)) \leq \liminf_{N \rightarrow +\infty} \varphi(\mu_N(T)).
\end{equation*}
Whence, $(\mu^*(\cdot),u^*(\cdot))$ is optimal for $(\Ppazo)$. 
\end{proof}


\section{Example of application}
\label{section:Example}

In this section, we illustrate the results of Sections \ref{section:DiffInc} and \ref{section:OCP} on an example of \textit{leader-follower} problem. This type of formulation frequently appears in the modelling of multi-agent systems (see e.g. \cite{MFPMP,Burger2020,FPR}). 

We will henceforth consider an optimal evacuation problem. Therein, a set of \textit{leaders} $\yb := \{y_i(\cdot)\}_{i=1}^M$ who can control their own accelerations $\ub := \{u_i(\cdot)\}_{i=1}^M$ via 
\begin{equation*}
\dot y_i(t) = w_i(t), \qquad \dot w_i(t) = u_i(t), 
\end{equation*}
aim at maximising the fraction of a crowd $\mu(\cdot)$ that reaches a closed target safety area $\Spazo \subset \R^d$ at time $T >0$,
\begin{equation*}
\max_{\ub(\cdot) \in \U} \INTDom{\mathds{1}_{\Spazo}(x)}{\R^{2d}}{\mu(T)(x,v)}. 
\end{equation*}
This type of final-time density maximisation has already been considered e.g. in \cite{Pogodaev2016}. 

We suppose that the controls $\ub(\cdot)$ have values in the state-dependent admissible sets
\begin{equation}
\label{eq:MixedControl}
U(\mu,\yb) = \Big\{ \ub \in (\R^d)^M ~\text{s.t.}~ |u_i| \leq C \Big(1 - \eta(\rho \star \mu)(y_i) \Big) \Big\}, 
\end{equation}
where $C >0$, $\eta \in [0,1]$, $\rho \in C^{\infty}_c(\R^d)$ is a smooth unitary mollifier with compact support and $(\rho \star \mu)$ is the convolution operator defined by
\begin{equation*}
(\rho \star \mu)(x) := \INTDom{\rho(x-x')}{\R^{2d}}{\mu(x',v')}.
\end{equation*}
This constraint encodes a form of ``soft congestion'' effect on the leaders, which amounts for the fact that they cannot move as fast as they wish if they are closely surrounded by too many agents. Introducing the empirical measures $\nu_M(\cdot) := \tfrac{1}{M} \sum_{i=1}^M \delta_{(y_i(\cdot),w_i(\cdot))}$, the dynamics of the leaders can be rewritten as  
\begin{equation*}
\partial_t \nu_M(t) + \Div \Big( \Wpazo(t,\ub(t))\nu_M(t) \Big) = 0, 
\end{equation*}
with a driving velocity field $ \Wpazo : (t,\ub) \in [0,T] \times(\R^d)^M \rightarrow C^0(\R^{2d},\R^{2d})$ defined by $\Wpazo(t,\ub)(y,w) := (w,\ub)^{\top}$. 

The dynamics of the crowd is modelled by a curve of densities $\mu(\cdot) \in \AC([0,T],\Pcal_c(\R^{2d}))$, whose evolution follows the non-local continuity equation 
\begin{equation*}
\partial_t \mu(t) + \Div \Big( \Vpazo (t,\mu(t),\nu_M(t)) \mu(t) \Big) = 0.
\end{equation*}
Here, the non-local velocity field $\Vpazo:(t,\mu,\nu) \in [0,T] \times \Pcal_c(\R^d) \times \Pcal_c(\R^d) \rightarrow C^0(\R^{2d},\R^{2d})$ writes
\begin{equation}
\label{eq:VelocityExample}
\Vpazo(t,\mu,\nu)(x,v) := \begin{pmatrix}
v \\
\Big( \Phi \star \mu(t) + \phi \star \nu(t) \Big)(x,v)
\end{pmatrix},
\end{equation}
where $\phi(\cdot),\Phi(\cdot)$ are defined respectively by 
\begin{equation*}
\left\{
\begin{aligned}
\phi(x,v) & := - \tfrac{K}{(\sigma + |x|)^{2\beta}} v , \\
\Phi(x,v) & := R_1 \exp \big( \hspace{-0.1cm} - \hspace{-0.1cm} \tfrac{|x|}{R_2} \big)v - A_1 \exp \big( \hspace{-0.1cm} - \hspace{-0.1cm} \tfrac{|x|}{A_2} \big)v,
\end{aligned}
\right.
\end{equation*}
for given constants $K, \sigma,\beta,R_1,R_2,A_1,A_2 \geq 0$, and $(\star)$ is the convolution operator. The acceleration term $(\Phi \star \mu)$ in \eqref{eq:VelocityExample} is the derivative of a \textit{Morse potential} (see e.g. \cite{Carrillo2013}), which encodes short range repulsions and long range attractions amongst the agents, while $(\phi \star \nu_M)$ is a \textit{Cucker-Smale} type kernel (see \cite{CS1}) which enforces the alignment of the velocities of $\mu(\cdot)$ with that of $\nu_M(\cdot)$.

We also impose an extra constraint that the crowd of agents must remain within a closed safety region $\Hpazo \subset \R^d$ during the whole evacuation process. This running constraint can be encoded in the spirit e.g. of \cite{Cavagnari2018} by functionals inequalities of the form
\begin{equation*}
\Lambda(\mu(t)) := \INTDom{d_{\Hpazo}(x)}{\R^{2d}}{\mu(t)(x,v)} \, \leq 0, \\
\end{equation*}
for all times $t \in [0,T]$, where $x \in \R^d \mapsto d_{\Hpazo}(x)$ denotes the Euclidean distance from $x \in \R^d$ to $\Hpazo \subset \R^d$.

We therefore consider the optimal control problem
\begin{equation*}
(\Ppazo) ~ \left\{
\begin{aligned}
\max_{\ub(\cdot) \in \U} & \INTDom{\mathds{1}_{\Spazo}(x)}{\R^{2d}}{\mu(T)(x,v)} \\
\text{s.t.} ~~ & \left\{
\begin{aligned}
& \partial_t \mu(t) + \Div \Big( \Vpazo(t,\mu(t),\nu_M(t)) \mu(t) \Big) = 0, \\ 
& \mu(0) = \mu^0, \\
& \partial_t \nu_M(t) + \Div \Big( \Wpazo(t,\ub(t)) \nu_M(t) \Big) = 0, \\
& \nu_M(0) = \tfrac{1}{M} \mathsmaller{\sum_{i=1}^M} \delta_{(y_i^0,w_i^0)}, 
\end{aligned}
\right. \\
\text{and} ~~ & \left\{
\begin{aligned}
& \ub(t) \in U(\mu(t),\nu_M(t)), \\
& \Lambda(\mu(t)) \leq 0,
\end{aligned}
\right.
\end{aligned}
\right.
\end{equation*}
where $U(\mu,\nu_M)$ is defined as in \eqref{eq:MixedControl} by identifying the empirical measure $\nu_M$ with the set of points $\yb \in (\R^d)^M$ where it is supported. It can be verified that all the objects involved in $(\Ppazo)$ satisfy hypotheses \ref{hyp:OCP} of Section \ref{section:OCP}. Therefore, problem $(\Ppazo)$ has an optimal solution. 


\section{Future work}

In the future, we aim at applying our set-valued approach of mean-field control to several topics. We are currently investigating first and second order necessary optimality conditions in the spirit of \cite{Frankowska2018}. We also aim at studying the \textit{sensitivity relations} (see e.g. \cite{Cannarsa1991}) which relate the Pontryagin costate of an optimal control problem to the super-differential of the value function.


{\footnotesize
\begin{flushleft}
\textbf{Acknowledgements :} ``This material is based upon work supported by the Air Force Office of Scientific Research under award number FA9550-18-1-0254.''
\end{flushleft}
}

{\footnotesize
\bibliographystyle{plain}
\bibliography{../ControlWassersteinBib}

\end{document}